\documentclass[a4paper,11pt]{amsart}

\usepackage{amssymb,amsfonts,amsthm,amscd,stmaryrd}

\usepackage{ae}
\usepackage[utf8]{inputenc}
\numberwithin{figure}{section}

\usepackage{mathrsfs,a4wide}
\usepackage{esint}

\usepackage{import}

\usepackage[leqno]{amsmath}

\usepackage[final,allcolors=blue,colorlinks=true]{hyperref}

\numberwithin{figure}{section}

\newtheorem{theorem}{Theorem}[section]
\newtheorem{lemma}[theorem]{Lemma}
\newtheorem{proposition}[theorem]{Proposition}
\newtheorem{corollary}[theorem]{Corollary}
\theoremstyle{definition}
\newtheorem{definition}[theorem]{Definition}

\numberwithin{equation}{section}

\newcommand{\R}{\mathbb{R}}

\newcommand{\Ha}{\mathcal{H}}

\newcommand{\eps}{\varepsilon}
\newcommand{\vphi}{\varphi} 
\newcommand{\la}{\langle}
\newcommand{\ra}{\rangle}
\newcommand{\diver}{\operatorname{div}}

\begin{document}

\title[Generalized Harnack inequality for semilinear  equations]{Generalized Harnack inequality for semilinear elliptic equations}

\author{Vesa Julin}
\address{University of Jyv\"{a}skyl\"{a},
Department of Mathematics and Statistics,
P.O.Box 35  FI-40014, Finland}
\email{vesa.julin@jyu.fi}


\begin{abstract} This paper is concerned with  semilinear  equations  in divergence form 
\[
\diver(A(x)Du) = f(u)
\]
where  $f :\R \to [0,\infty)$ is nondecreasing. We introduce a sharp Harnack type inequality for nonnegative solutions  which is a  quantified version of the condition  for strong maximum principle found by Vazquez and Pucci-Serrin in   \cite{Vaz, PS2000}  and   is closely related to the classical Keller-Osserman condition \cite{Keller, Oss} for the existence of entire solutions.

\end{abstract}

\maketitle

\section{Introduction}

In this paper we study nonnegative solutions of the equation
\begin{equation} \label{the pde}
\diver(A(x)Du) = f(u).
\end{equation}
The coefficient  matrix $A(x)$ is assumed to be symmetric, measurable  and to satisfy the uniform ellipticity condition 
\[
\lambda |\xi|^2 \leq \la A(x) \xi, \, \xi \ra  \leq  \Lambda |\xi|^2
\]
for every $\xi \in \R^n$, where $0< \lambda \leq \Lambda$. The function $f : \R \to [0,\infty)$ is assumed to be nonnegative and nondecreasing. Note that we allow $f$ to have jump discontinuity.  Throughout the paper we denote the integral function of $f$ by $F$, i.e., 
\[
F(t) = \int_0^t f(s)\, ds. 
\]

The goal of this paper is to prove a general Harnack inequality for \eqref{the pde}. To be more precise we seek to answer the following simple question in a quantitative way. If $u$ is a nonnegative solution of \eqref{the pde} in $B_2$, then does 
the value $\inf_{B_1}u$ control $\sup_{B_1}u$?  I would like to  stress that this paper does not concern regularity of solutions of \eqref{the pde}. In fact,   the regularity for \eqref{the pde} is well understood.  Indeed, since $f$ is nonnegative,  any solution of \eqref{the pde} is a weak  subsolution of the corresponding linear equation.
 Therefore by  De Giorgi   theorem \cite{degiorgi}  nonnegative solutions are  locally bounded and  by the De Giorgi-Nash-Moser theorem they are H\"older continuous. 
However, the point is that  both the  $L^\infty$-bound and the H\"older norm depend on the $L^2$-norm of the solution  
and it is not clear how one can  bound the $L^2$-norm  by knowing the   value of the solution only   at one point. 

The main result reads as follows.
\begin{theorem} \label{thm1}
Let $u \in W^{1,2}(B_2)$ be  a nonnegative solution of \eqref{the pde}. Denote $M= \sup_{B_1}u$ and $m= \inf_{B_1}u$. There is a constant $C$ which depends only on the ellipticity constants $\lambda, \Lambda$ and the dimension $n$ such that 
\begin{equation} \label{harnack oma}
\int_m^M \frac{dt}{\sqrt{F(t)} +t } \leq C.
\end{equation}
In particular, $C$ is independent of the solution $u$ and of the function $f$.
\end{theorem}
Theorem \ref{thm1} is in the same spirit as  \cite{Vesku} for nondivergence form equations with nonhomogeneous gradient drift term. The above result is stronger and more complete  than the main result in \cite{Vesku}, because  we do not need any regularity nor growth assumptions on $f$, and most importantly, the  constant in Theorem \ref{thm1} is independent of $f$. In particular, the estimate is stable under scaling.

Harnack inequality for linear elliptic equations by Moser \cite{Moser} is one of the most important results in the theory of elliptic partial differential equations. There are numerous generalizations of this theorem from which the most relevant for us  is by Di Benedetto-Trudinger \cite{DT} who proved  the Harnack inequality   for quasiminimizers of integral functionals. Harnack inequality for general  quasilinear equations  has been considered e.g. by Serrin \cite{serrin} (see also  \cite{LU} for the H\"older continuity). This result has the disadvantage that the constant will depend on the solution itself. In  Theorem \ref{thm1}  the inequality   \eqref{harnack oma} is not in the classical form but its structure depends on the scaling of the equation. This has the advantage that the  constant is then independent of the solution. 
Theorem \ref{thm1} is  more similar to  Pucci-Serrin \cite{PS2001} who introduced a Harnack inequality  in $\R^2$ for  quasilinear equations similar to  \eqref{the pde} where the operator is allowed to be  nonlinear but not to have  dependence on $x$.    Compared to \cite{PS2001} the advantage of  Theorem \ref{thm1}  is that the dependence on every  parameter in \eqref{harnack oma} is explicit and we do not need any further  assumptions on   $f$ other than the  monotonicity. This makes the result more general and the esimate \eqref{harnack oma} more stable. Moreover, the result holds in any dimension and the operator is allowed to have merely bounded coefficients. In particular, it is not possible to prove Theorem \ref{thm1} by a comparison argument.    
 
The drawback of  \eqref{harnack oma} is that it is in  implicit form and   it may be difficult to write excplicitly the  relation between the maximum  and the minimum. On the other hand     \eqref{harnack oma}  is a natural generalization of the classical Harnack inequality  for equations of type \eqref{the pde}. We illustrate this  by giving a complete answer to  the following problems.  
In the following $u \in W^{1,2}(B_2)$ is a nonnegative solution of \eqref{the pde} with $M=  \sup_{B_1}u $ and $m = \inf_{B_1}u$.
\begin{itemize} 
\item[(i)]  \textbf{Strong minimum principle}. If $u$ is zero at one point in $B_1$, is it zero everywhere? 
\item[(ii)] \textbf{Boundedness}. Is $u$  bounded  in $B_1$  by a constant which depends only on $u(0)$? 
\item[(iii)] \textbf{Local Boundedness}.  Is there a radius $r>0$ such that  $u$ is bounded  in $B_r$  by a constant which depends only on $u(0)$?
\end{itemize} 

The problem  (i) has been considered  by Vazquez \cite{Vaz} and by Pucci-Serrin and their collaborators \cite{PS1999, PS2000, PS2004, PS2007} and we know that the strong  minimum principle holds if $f$ is positive and
\begin{equation}\label{min princ}
\int_0^1 \frac{dt}{\sqrt{F(t)}} = \infty.
\end{equation}
In fact, this condition is also necessary \cite{BBC, diaz}. Theorem \ref{thm1} is in accordance with this and it provides a quantification of the strong minimum principle. By this we mean that if 
\[
\int_0^1 \frac{dt}{\sqrt{F(t)}+t} = \infty,
\]
then  there is a continuous,   increasing  function $\Phi :[0,1] \to [0,1]$ such that for $M \leq 1$ it holds  $ m \geq \Phi(M)$,  (see also  \cite{PS2001}). In order to quantify the strong minimum principle we need to replace the function $\sqrt{F(t)}$ in  \eqref{min princ}  by $\sqrt{F(t)} +t$. This is natural since  in the linear case $f \equiv 0$ the estimate \eqref{harnack oma} then  reduces to the classical Harnack inequality.

Similarly we find an answer to  (ii).  If  
\[
\int_1^\infty \frac{dt}{\sqrt{F(t)}+t} = \infty,
\]
 then  there is a continuous,   increasing  function $\Psi :[1,\infty) \to [1,\infty)$ such that for $m \geq 1$ it holds  $ M \leq \Psi(m)$,  (see also  \cite{PS2001}).
This condition is very similar to the Keller-Osserman condition \cite{Keller, Oss} which states that if $f$ is positive and  $u$ is a nonnegative and nontrivial  solution of \eqref{the pde} in the whole $\R^n$ then necessarily
\begin{equation} \label{keller oss 2}
\int_1^\infty \frac{dt}{\sqrt{F(t)}} = \infty.
\end{equation}  
For the proof of this  see \cite{Konkov} or Theorem \ref{theorem sub} below.  The above condition is sharp for the uniform boundedness of the solutions. Indeed if $f$ does not satisfy \eqref{keller oss 2} then there exists a sequence of nonnegative solutions $(u_k)$ of \eqref{the pde} such that $u_k(0)\leq 1$ and $\sup_{B_1}u_k \to \infty $ as $k \to \infty$. We leave this to the reader.

Finally to answer (iii) we find that if $u$ and $f$ are  as in Theorem \ref{thm1}, then  we may always  find a radius $r>0$ such that $\sup_{B_r}u$ is uniformly bounded by a constant which depends on the value $u(0)$. Indeed, since the constant in Theorem \ref{thm1} does not depend on $u$ and $f$, we obtain by a simple scaling argument that for $M_r = \sup_{B_r}u$ it holds
\[
\int_{u(0)}^{M_r} \frac{dt}{r \sqrt{F(t)} +t } \leq C.
\]
Note that when $r \to 0$ the above estimate converges to the classical Harnack inequality.  Therefore when $r>0$ is small enough we have that $M_r < \infty$. This result can be seen as a weak counterpart of the short time existence result for the one dimensional   initial value  problem
\[
y''= f(y), \qquad y(0) = y_0 \quad \text{and} \quad y'(0) = y_0'.
\] 
Note that if $f$ satisfies \eqref{keller oss 2}  the above initial value  problem    has a solution in the whole $\R$.

The statement of Theorem \ref{thm1} is sharp which can be seen already in dimension one (see \cite[Remark 1]{Oss}). The assumption $f\geq 0$ is also necessary, i.e., Theorem \ref{thm1} is not true for equations 
\[
-\Delta u = f(u)
\] 
where $f$ is nonnegative and monotone. This can be seen by a simple example which we give at the end of the paper. One reason for this is  that the above equation is the Euler equation of the nonconvex functional
\[
\int_{B_2} \frac{1}{2}|Du|^2\, dx - F(u)\, dx
\]
and criticality alone is not enough to prove the optimal  $L^\infty$- bound for  minimizers (see \cite{cabre} and the references therein). On the other hand  Theorem \ref{thm1}  could still hold without the monotonicity  assumption on $f$.

The proof of Theorem  \ref{thm1} is rather long and has several  stages, and therefore we give its outline  here.  In this paper we develop further the ideas from \cite{Vesku}. The main difficulty is to  overcome the lack of regularity 
and growth condition on $f$, and to avoid the constant $C$ to depend on $f$. In order to do this we will  revisit the proof of the classical Harnack inequality by Di Benedetto-Trudinger \cite{DT} in order  to have a more suitable and sharper version
which allows us to treat  the nonhomogeneous case. 

To overcome the lack of regularity and to have the constant independent of $f$, we   use the fact that the equation is in divergence form and integrate it locally over the level sets of the solution  (as in  \cite{mazya, Talenti}). This will give us precise  information how fast the level sets  locally decay.  Similar method has been used to study global regularity  for solutions of elliptic equations e.g. in \cite{CM}. Here we use it to prove local estimates. 

The first  observation is that we have a good local estimate on the decay rate of the level set $\{u \geq t\}$ when we are  in a ball  $B(x,r)$ whose radius is small $r \simeq t/\sqrt{F(2t)}$ and the density of the level set in the ball
\[
\sigma_t(x)= \frac{|\{u \geq t\} \cap B(x,r)|}{|B(x,r)|}
\]
  is not close to one or zero (Lemma \ref{caccio invert}). The second  observation is a measure theoretical lemma (Lemma \ref{isop ineq 1}) which states that the measure of the set where the density $\sigma_t$ is between $1/5$ and $4/5$  (which can be thought to be the ''boundary'' of the level set $\{ u\geq t\}$) is related to the  measure of the set where the density is larger than $4/5$ (which corresponds to the ''interior'' of the level set   $\{ u\geq t\}$). This  is of course very much related to the isoperimetric inequality. We combine these two lemmas and obtain the following   estimate (see Proposition \eqref{decay oma}) for the decay rate  of the level set $\mu(t) = |\{u \geq t\} \cap B_2|$,  
\begin{equation} \label{super intro}
 -\mu'(t) \geq c \min \Bigl\{  \frac{1}{\sqrt{F(2t)}}  \mu(t)^{\frac{n-1}{n}} , \frac{1}{t}\mu(t)\Bigl\},
\end{equation}
for almost every $t>0$ for which $\mu(t) \leq |B_2|/2$. When $\mu(t) \geq |B_2|/2$ we have a similar estimate for  $\eta(t) =  |\{u < t\} \cap B_2|$.  Note that \eqref{super intro} has two parts on the right hand side. The first one is the nonhomogeneous estimate and the second is the homogeneous one. In the sublinear case $f(t) \leq t$ we may integrate \eqref{super intro} and conclude that  the solution is $L^\eps$-intergable.  In the nonhomogeneous case  the inequality  \eqref{super intro}  may oscillate between the two estimates.

Another  issue in the proof is to overcome the fact that we do not have any growth condition on $f$. In \cite{Vesku} it was assumed that the nonhomogeneity is of type $f(t)= g(t)t$, where $g$ is a slowly increasing function. The proof was based on the idea that under this assumption any unbounded supersolution blows up as the fundamental solution of the linear equation. This is certainly not true in our case. To solve this problem we study more closely subsolutions of \eqref{the pde} and prove an estimate (Lemma \ref{de giorgi subille}) which roughly speaking quantifies the  fact that if $f$ is very large  then  solutions of \eqref{the pde} will  grow fast compared to  the solutions of the corresponding linear equation. We iterate Lemma \ref{de giorgi subille} and obtain the following lower bound for subsolutions. Similar result is obtained also in \cite{Konkov}.
\begin{theorem}
\label{theorem sub}
Let   $u \in W^{1,2}(B(x_0,2R))$ be a continuous and  nonnegative subsolution of
\[
\diver(A(x)Du) \geq f(u)
\]
and denote $M = \sup_{B(x_0,R)} u$ and $m = \inf_{B(x_0,R)} u$. If $u(x_0)>0$ then it holds
\[
\int_{m/4}^M \frac{dt}{\sqrt{F(t)}} \geq c \, R
\] 
for a  constant $c > 0$ which is independent of $u$ and $f$.
\end{theorem}
In fact, we will not need this result  in the proof of  Theorem \ref{thm1} but only Lemma \ref{de giorgi subille}. However since  Theorem \ref{theorem sub} follows rather easily from  Lemma \ref{de giorgi subille} we choose to state it. At the end of  Section 4 in  Corollary \ref{keller osserman oma} we show that Theorem \ref{theorem sub} implies the Keller-Osserman condition for entire solutions of \eqref{the pde}. This result has been proved in \cite{Konkov}. This result is also known for wide class of nondivergence form operators  \cite{CDLV1, CDLV2, FQS}.

Let us now briefly  give a rough version of the proof of  Theorem  \ref{thm1}. We integrate \eqref{super intro} (and its counterpart for $\eta$) and  conclude that for every $\eps>0$  there exists $t_\eps>0$ such that $\mu(t_\eps) =  |\{u \geq t_\eps\} \cap B_2| \leq \eps$ and 
\[
\int_{m}^{t_\eps} \frac{dt}{\sqrt{F(2t)} +t} \leq C_\eps.
\]
For simplicity assume that for every $t >t_\eps$ we have   the  nonhomogeneous (the first) estimate in \eqref{super intro}. Then  integrating \eqref{super intro} gives 
\[
c \int_{t_\eps}^{\infty} \frac{dt}{\sqrt{F(2t)}} \leq  -  \int_{t_\eps}^{\infty} \mu^{\frac{1}{n}-1}(t)\mu'(t)\, dt =   \frac{1}{n} \mu^{\frac{1}{n}}(t_\eps) \leq \frac{1}{n}\eps^{\frac{1}{n}}. 
\]
We conclude that  it has to hold $M = \max_{B_{1}}u \leq 8 t_\eps $. Indeed,  otherwise Theorem \ref{theorem sub} would imply for  $M_{3/2}= \max_{B_{3/2}}u$     that 
\[
\frac{c}{2} \leq \int_{2t_\eps}^{M_{3/2}} \frac{dt}{\sqrt{F(t)}} \leq  2 \int_{t_\eps}^{M_{3/2}} \frac{dt}{\sqrt{F(2t)}}
\]
which is  a contradiction when $\eps$ is small. Hence we have 
\[
\int_{m}^{M} \frac{dt}{\sqrt{F(2t)} +t} \leq C, 
\]
which implies the result by  change of variables.

The paper is organized as follows. In the next section we recall basic results from measure theory. In Section 3 we prove estimates for
nonnegative supersolutions of \eqref{the pde}. The main result of that section is  Proposition \ref{decay oma}.  In Section 4 we prove estimates for subsolutions of \eqref{the pde} and prove Theorem \ref{theorem sub} and show in Corollary  \ref{keller osserman oma} 
how it implies the Keller-Osserman condition for entire solutions. Finally in Section 5 we give the proof of  Theorem \ref{thm1}.

\section{Preliminaries}

Throughout the paper we denote by $B(x,r)$ the open  ball centred at $x$ with radius $r$. In the case $x=0$ we simply write $B(0,r) =B_r $.   

Let $U \subset \R^n$ be an open set. A set  $E \subset \R^n$ has finite perimeter in $U$ if 
\[
P(E,U) := \sup \left\{ \int_E \diver \vphi \, dx \, : \, \vphi \in C_0^1(U, \R^n), \,\, ||\vphi||_{\infty} \leq 1 \right\} < \infty.
\]
Here $P(E,U)$ is the perimeter of $E$ in $U$. Sometimes we call $E$ a set of finite perimeter if it is clear from the context which is the reference domain  $U$. 
Let $E$ be a set of finite perimeter in $U$. The reduced boundary of $E$ is denoted by $\partial^* E$. It is smaller than  the topological boundary 
 which we denote by $\partial E$. For any open set $V \subset U$ it holds $P(E,V) = \int_{\partial^* E \cap V} \, d\Ha^{n-1}$, where $\Ha^{n-1}$
is the standard $(n-1)$-dimensional Hausdorff measure. Moreover the Gauss-Green formula holds
\[
 \int_E \diver X  \, dx =  \int_{\partial^* E} \la X, \nu \ra\, d\Ha^{n-1}
\] 
for every $X \in W_0^{1, \infty}(U, \R^n)$. Here $\nu$ is the outer unit normal of $E$ which exists on $\partial^* E$. 
For an introduction to the theory of sets of finite perimeter we refer to \cite{AFP} and \cite{Maggi}. All the following results can be found in these books.    

The most important result from geometric measure theory for us  is the relative isoperimetric inequality. It states that for every set of finite perimeter
$E$ in the ball $B_R$ the following inequality holds
\[
P(E,B_R) \geq c \min \left\{ |E\cap B_R|^{\frac{n-1}{n}}, |B_R \setminus E|^{\frac{n-1}{n}}  \right\}
\]
for a constant $c$ which depends on the dimension. The proof of the main result is mostly based on this inequality. 

We recall the coarea formula for Lipschitz functions. Let $g$ be Lipschitz continuous in an open set $U$ and let $h \in L^1(U)$ be nonnegative. 
Then it holds
\[
\int_U h(x) |Dg(x)|\, dx =  \int_{-\infty}^\infty \left( \int_{\{g=t\}  \cap U} h(x) d\Ha^{n-1}(x) \right) dt.
\]
The formula still  holds if  $g$ is locally Lipschitz continuous and $h |Dg| \in  L^1(U)$. From the coarea formula
one deduces immediately that almost every level set  $\{ g>t \}$  of a Lipschitz function is a set
of finite perimeter.  In the  case $g \in C^\infty(U)$ the level sets are even more regular, since by the Morse-Sard Lemma 
the image of the critical set $K = \{x \in U :  |Dg(x)| = 0 \}$ has measure zero $|g(K)| =0$. In particular, almost every level set
of a smooth function has  smooth boundary. 

Throughout the paper we denote the sublevel sets of a measurable  function $u:B_2 \to \R$  by
\[
E_t := \{x \in B_2 : u(x) < t \}
\]   
and the superlevel sets by 
\[
A_t := \{x \in B_2 : u(x) \geq t \}.
\]  
Moreover we denote $\mu(t)= |A_t|$ and $\eta(t)= |E_t|$. In Section 3 we estimate the differential  of $\mu(t)$ (and $\eta(t)$) when $u$
is nonnegative supersolution of the equation \eqref{the pde}. The differential of  $\mu(\cdot )$  is a measure and we denote its absolutely continuous part by  $\mu'$.
 To avoid pathological situations (see \cite{AL}) we  regularize the equation in order to work with smooth functions.
Then by the result from \cite{AL} we may write  $\mu'$ as 
\[
\mu'(t) = - \int_{\{ u=t\} \cap B_2} \frac{1}{|Du|}\, d\Ha^{n-1}
\] 
for almost every $t$.

Let us turn our attention to  elliptic equations in divergence form. As an introduction to the topic we refer to \cite{Giusti}.  
\begin{definition}\label{def solution}
A function  $u \in W^{1,2}(U)$ is a  supersolution   of \eqref{the pde} in an open set  $U$ if
\[
\int_{U} \la A(x) Du, D\vphi \ra\, dx \geq -\int_{U}f(u) \vphi\, dx \qquad \text{for all nonnegative }\, \vphi \in W_0^{1,2}(U).
\]
A function  $u \in W^{1,2}(U)$ is a subsolution of \eqref{the pde} in $\Omega$ if
\[
\int_{U} \la A(x) Du, D\vphi \ra\, dx \leq -\int_{U}f(u) \vphi\, dx \qquad \text{for all nonnegative } \, \vphi \in  W_0^{1,2}(U).
\]
Finally   $u \in W^{1,2}(U)$ is a solution of \eqref{the pde} in $U$ if it is both super- and subsolution. 
\end{definition}

As we already mentioned we will regularize the equation \eqref{the pde} and work with solutions of
\begin{equation} \label{pde reg}
\diver(A_\eps(x) D v) =  f_\eps(v)
\end{equation}
where $A_\eps$ is smooth, symmetric and elliptic, and $ f_\eps$ is positive, increasing and smooth. Moreover we choose $A_\eps$ and $f_\eps$ such that   $A_\eps \to A$ in $L^1$ 
and $f_\eps \to f$ in $L_{loc}^1(\R)$ such that $f_{\eps_2} \geq f_{\eps_1} $ for $\eps_2 > \eps_1$. Let $u \in W^{1,2}(B_2)$ be a nonnegative solution of \eqref{the pde} in $B_2$. For every $\eps>0$ we define $u_\eps$
to  be  the solution of \eqref{pde reg} having the same boundary values as  $u$, i.e.,  $u_\eps - u \in W_0^{1,2}(B_2)$. 
Since $(u_\eps)$ are locally uniformly H\"older continuous we have that $u_\eps \to u$  
uniformly in $B_1$ and (denote $F_\eps(t) = \int_0^t f_\eps(s)\, ds$, $M_\eps = \sup_{B_1} u_\eps$ and $m_\eps = \inf_{B_1} u_\eps$) 
\[
\lim_{\eps \to 0} \int_{m_\eps}^{M_\eps} \frac{dt}{\sqrt{F_\eps(t)} +t } = \int_{m}^{M} \frac{dt}{\sqrt{F(t)} +t }
\]
by monotone convergence. Hence, we may assume that the solution $u$ in Theorem \ref{thm1} is smooth, i.e., $u \in C^\infty(B_2)$.

Finally for  De Giorgi iteration  we recall the following lemma which can be found e.g. in  \cite[Lemma 7.1]{Giusti}.
\begin{lemma}
\label{algebra lemma}
Let $(x_i)$ be a sequence of positive numbers such that 
\[
x_{i+1} \leq C_0 B^i x_i^{1+ \frac{1}{n}}
\]
with $C_0>0$ and $B>1$. If $x_0 \leq C_0^{-n} B^{-n^2}$  then 
\[
\lim_{i \to \infty} x_i = 0.
\]
\end{lemma}

\section{Estimates for supersolutions}

In this section we study nonnegative supersolutions of \eqref{the pde}. 
We begin by proving a standard  Caccioppoli inequality and its  variant, which involves a boundary term. As we mentioned in the previous section 
we prefer to work with smooth functions. Recall that $E_t = \{ u <t\} \cap B_2$.

\begin{proposition}
\label{cacccioppoli}
Assume  that $u \in C^{\infty}(B_{R})$ is a nonnegative supersolution of 
\begin{equation} \label{super eq}
\diver(A(x)Du) \leq f(u).
\end{equation}
Then for $r <R$  it holds 
\[
\int_{E_t \cap B_r}|Du|^2\, dx \leq C\left(\frac{t^2}{(R-r)^2}  + F(2t)\right) |E_t\cap B_{R}|
\]
for a constant depending on $\lambda$ and $\Lambda$. Moreover for almost every $t>0$ we have
\[
\int_{\{ u=t \}\cap B_r} |Du| \, d \Ha^{n-1} \leq    C\left(  \frac{t}{(R-r)^2} + \frac{F(2t)}{t} \right) |E_t \cap B_R|.
\]

\end{proposition}

\begin{proof}
 Let us fix  $t >0$  and choose a testfunction $\vphi(x)= (t-u(x))_+ \zeta^2(x)$, where $\zeta \in C_0^1(B_R)$ is a standard cut-off function such that 
$\zeta(x)=1$ in $B_r$ and $|D\zeta|\leq \frac{2}{R - r}$. We get after using the  ellipticity and  Young's inequality that 
\[
\int_{E_t \cap B_r}|Du|^2\, dx \leq  \frac{C}{(R-r)^2} \int_{B_R} (t-u)_+^2\, dx + \int_{B_R} f(u)(t-u)_+\, dx.
\]
Since $f$ is nondecreasing  we have $f(t)t \leq \int_t^{2t} f(s)\, ds \leq F(2t)$. Therefore  the standard Caccioppoli inequality follows from 
\begin{equation} \label{estimoi f}
\begin{split}
 \int_{B_R} f(u)(t-u)_+\, dx  \leq f(t)t |E_t\cap B_{R}| \leq  F(2t)|E_t\cap B_{R}|. 
\end{split}
\end{equation}

By Morse-Sard Lemma it holds $|D u| > 0$ on $\{ u = t \}$ for almost every $t>0$. Let us prove that  the second statement holds for every such $t$.  We integrate the equation \eqref{super eq} over the set $E_t \cap B_\rho$  and get
\[
\int_{\{ u=t \} \cap B_\rho} \la A(x) Du, \frac{Du}{|Du|} \ra \, d \Ha^{n-1} \leq -\int_{\partial B_\rho \cap E_t}  \la A(x) Du, \frac{x}{|x|} \ra \, d \Ha^{n-1} + \int_{E_t \cap B_\rho} f(u)\, dx.
\]
For a rigorous argument see \cite{Talenti} and note that (see \cite{EG})
\[
\frac{d}{dt} \int_{\{u <t\} \cap B_\rho}  \la A(x) Du, Du \ra \, dx = \int_{ \{u = t\} \cap B_\rho} \la A(x) Du, \frac{Du}{|Du|} \ra \, d \Ha^{n-1}.
\]
 By the ellipticity  we have  
\[
\int_{\{ u=t \}\cap B_\rho} |Du| \, d \Ha^{n-1} \leq C\int_{\partial B_\rho \cap E_t} |Du|\, d \Ha^{n-1} + C\int_{E_t \cap B_\rho} f(u)\, dx.
\]
Choose $\hat{\rho}= \frac{r+R}{2}$ and integrate the previous inequality with respect to $\rho$ over $(r, \hat{\rho})$ and get
\begin{equation} \label{caccio 2}
(R-r) \int_{\{ u=t \}\cap B_r} |Du| \, d \Ha^{n-1} \leq C\int_{B_{\hat{\rho}} \cap E_t} |Du|\, dx + C(R - r)\int_{E_t \cap B_R} f(u)\, dx.
\end{equation}
By the previous Caccioppoli inequality  we have
\[
\begin{split}
\int_{E_t \cap B_{\hat{\rho}}} |Du|\, dx &\leq |E_t \cap B_{\hat{\rho}}|^{\frac{1}{2}}\left(\int_{E_t \cap B_{\hat{\rho}}} |Du|^2\, dx\right)^{\frac{1}{2}}\\
&\leq  C\left(\frac{t}{(R-r)}+ \sqrt{F(2t)} \right) |E_t \cap B_R|.
\end{split}
\]
Arguing as in \eqref{estimoi f} we get
\[
\int_{E_t \cap B_R} f(u)\, dx \leq  \frac{F(2t)}{t}|E_t\cap B_{R}|.
\]
Therefore from \eqref{caccio 2} we get by Young's inequality 
\[
\begin{split}
\int_{\{ u=t \} \cap B_r} |Du| \, d \Ha^{n-1}  &\leq  C\left(\frac{t}{(R-r)^2} +  \frac{\sqrt{F(2t)}}{R - r} + \frac{F(2t)}{t} \right) |E_t \cap B_R|\\
&\leq C\left(\frac{t}{(R-r)^2} + \frac{F(2t)}{t} \right) |E_t \cap B_R|.
\end{split}
\]
\end{proof}

Next we observe that if the measure of the sublevel set $|E_t \cap B_2|$ is very small compared to the ratio $t/\sqrt{F(2t)}$, then we may apply the estimates from  
the linear theory. The next result follows from the previous Caccioppoli inequality together with the standard  De Giorgi iteration. The argument is almost exactly 
the same as \cite[Lemma 7.4]{Giusti} but we give the proof for the reader's convenience.

\begin{lemma}
\label{DT1}
Assume that  $u \in C^{\infty}(B_{2})$ is a  nonnegative supersolution of
\[
\diver(A(x)Du) \leq f(u).
\]
There  is  $\delta_0>0$ such that if for some $t >0$ it holds
\[
|E_{t} \cap B_{2}| \leq \delta_0  \left( \frac{t}{\sqrt{F(2t)} +t} \right)^n
\]
then one has
\[
\inf_{B_1}u \geq \frac{t}{2}.
\]
\end{lemma}

\begin{proof}
Let $t>0$  satisfy the assumption of the lemma. For $0<h<k \leq t$ we define
\[
v=\begin{cases}
                        0,&\text{if $u\geq k$}\\
k-u,&\text{if $h<u<k$}\\
k-h,&\text{if $u \leq h$}.
                       \end{cases}
\]
By possibly decreasing the value of $\delta_0$ we  may assume that  $|E_{t} \cap B_{2}| \leq \frac{1}{2}|B_{1}|$. Therefore by the Sobolev inequality we have  for every $1 \leq \rho <R \leq 2$ that
\[
(k-h)|E_h \cap B_\rho|^{\frac{n-1}{n}}  \leq  \left( \int_{B_\rho} v^{\frac{n}{n-1}}\, dx\right)^{\frac{n-1}{n}} \leq C  \int_{E_k \cap B_\rho } |Du|\, dx.
\] 
On the other hand by  Proposition \ref{cacccioppoli} we have
\[
 \int_{E_k \cap B_\rho } |Du|\, dx \leq |E_k \cap B_\rho |^{\frac12} \left( \int_{E_k \cap B_\rho } |Du|^2\, dx  \right)^{\frac12} \leq C \left(\frac{k}{R-\rho}+  \sqrt{F(2k)} \right) |E_k \cap B_R|.
\]
These two estimates yield
\[
(k-h)| E_h \cap B_\rho|^{ \frac{n-1}{n}} \leq C\left(\frac{k}{R-\rho}+  \sqrt{F(2k)} \right) |E_k \cap B_R|
\]
for every $1 \leq \rho <R \leq 2$ and $0<h<k \leq t$.

Define  $r_i = (1+ 2^{-i})$ and $k_i = \frac{t}{2}(1 + 2^{-i})$ and apply the above  inequality  for $R= r_i$, $\rho = r_{i+1}$, $h = k_{i+1}$ and $k = k_{i}$.  We set $x_i =| E_{k_i} \cap B_{r_i}|$ and 
obtain
\[
x_{i+1}\leq C_0 \left(1 + \frac{\sqrt{F(2t)} }{t} \right)  4^{i}  x_i^{1+ \frac{1}{n}} .
\]
We choose $\delta_0 >0$ such that 
\[
\delta_0 \leq C_0^{-n}4^{-n^2}.
\]
By the assumption on $t$ we have 
\[ 
x_0 = |E_{t} \cap B_{2}| \leq \delta_0  \left( \frac{t}{\sqrt{F(2t)} +t} \right)^n
\]
and we conclude from  Lemma \ref{algebra lemma} that   $\lim_{i \to \infty }x_i = 0$. In other words $\inf_{B_1} u \geq t/2$.
\end{proof}

We turn our attention  to Proposition \ref{decay oma} which is the main result of this section. To that aim we will need  two  lemmas. 
In the first  lemma we use the Caccioppoli inequality to study the local decay rate of  the  level sets $E_t$. 
Due to the nonhomogeneity of the equation we  do this only in small balls $B(x,r)$ with radius  $r \leq t/\sqrt{F(2t)}$. Due to the relative isoperimetric inequality the estimate
depends  whether the density  $|E_t \cap B(x,r)|/|B_r|$ is close to one or close to zero.

\begin{lemma}
\label{caccio invert}
Assume that   $u \in C^{\infty}(B_{2})$ is a  nonnegative supersolution of 
\[
\diver(A(x)Du) \leq f(u).
\]
Then for almost every $t>0$ it holds
\[
 \int_{\{ u=t \} \cap B(x_0, 2r)}\frac{1}{|Du|}\, \Ha^{n-1} \geq \frac{c}{t}\,\min \Big{\{}  \big|E_t \cap  B(x_0,r) \big|,  \big|B(x_0,r) \setminus E_t \big|   \Big{\}}
\]
whenever  $B(x_0,2r) \subset  B_2$ is such that 
\[
r \leq  \frac{t}{\sqrt{F(2t)}}.
\]
\end{lemma}

\begin{proof}

Without loss of generality we may assume that $x_0 = 0$. We may also assume that $|E_t \cap  B(x_0,r)| \leq  |B(x_0,r) \setminus E_t |$ since in the other case the proof is similar. 
Moreover by Morse-Sard Lemma we may assume that $|Du|>0$ on $\{ u=t\} \cap B_2$.

Let us first assume that 
\begin{equation} \label{small case}
|E_t \cap  B_r| < 2^{-n-2} |B_r|.
\end{equation}
Let us fix $x_i \in E_t \cap B_r$.  For $x_i$ we define a radius $R_i$ such that 
\[
R_i := \inf \big{\{} \rho >0 : |B(x_i, \rho) \cap  E_t| \leq (1- 2^{-n-1}) |B(x_i, \rho)| \big{\}}.
\]  
Since $x_i \in E_t$ and $E_t$ is an open set we have  that $R_i >0$. Since the point $x_i$ is in $B_r$  then the ball  $B(\frac{x_i}{2}, \frac{r}{2})$ is in the intersection $B_r \cap B(x_i,r)$, and thus
\[
|B(x_i, r) \setminus B_r| \leq |B_r|- |B_{r/2}| = (1- 2^{-n})|B_r|.
\]  
Therefore we deduce from the assumption \eqref{small case}   that  
\[
\begin{split}
|B(x_i, r) \cap  E_t| &\leq |B(x_i, r) \setminus B_r| + |E_t \cap B_r| \\
&< (1- 2^{-n})|B(x_i, r) |+ 2^{-n-2} |B(x_i, r) |\\
&\leq  (1- 2^{-n-1}) |B(x_i, r)|.
\end{split}
\]
Thus  we conclude that  $R_i \leq r$.
Hence we obtain a family of  balls $\bar{B}(x_i, R_i)$ which cover $E_t \cap B_r$ and satisfy $0<R_i \leq r$. By the Besikovitch covering theorem \cite[Corollary 5.2]{Maggi} we may choose a countable 
 disjoint subfamily, say $\mathcal{F}$, such that    $\xi(n) \sum_{i \in \mathcal{F}}|B(x_i,R_i)| \geq  |E_t \cap B_r|$.
Moreover by the definition of $R_i$ it holds
\[
|E_t \cap B(x_i, R_i)| =  (1- 2^{-n-1})|B(x_i, R_i)| = c R_i^n.
\]
In particular, the density $|E_t \cap B(x_i, R_i)| /|B_{R_i}|$ is bounded away from zero and one. 

Let us fix a ball $B(x_i, R_i)$ in the Besikovitch cover.  We use   Proposition \ref{cacccioppoli} in $B(x_i, R_i)$ and obtain 
\begin{equation} 
\label{half density 1}
\begin{split}
\int_{\{u=t\} \cap B(x_i, R_i)} |Du| \, d \Ha^{n-1} &\leq    C\left(  \frac{t}{r^2} + \frac{F(2t)}{t} \right) |E_t \cap B(x_i, 2R_i) |\\
&\leq Ct \, R_i^{n-2},
\end{split}
\end{equation} 
where the last inequality follows from $R_i \leq r  \leq  t/\sqrt{F(2t)}$.
On the other hand  the relative isoperimetric inequality yields 
\begin{equation} 
\label{half density 2}
\begin{split}
\int_{\{u=t\} \cap  B(x_i, R_i)} |Du| \, d \Ha^{n-1} &\geq c \left( \int_{\{u=t\} \cap B(x_i, R_i)} \frac{1}{|Du|} \, d \Ha^{n-1}  \right)^{-1} P(E_t, B(x_i, R_i))^2\\
&\geq c \left( \int_{\{u=t\} \cap B(x_i, R_i)} \frac{1}{|Du|} \, d \Ha^{n-1}  \right)^{-1}|E_t \cap B(x_i, R_i)|^{\frac{2(n-1)}{n}}\\
&\geq c \left( \int_{\{u=t\} \cap B(x_i, R_i)} \frac{1}{|Du|} \, d \Ha^{n-1}  \right)^{-1}  R_i^{2n - 2}.
\end{split}
\end{equation} 
Therefore \eqref{half density 1} and \eqref{half density 2} give
\begin{equation} 
\label{low density 1}
 \int_{\{u=t\} \cap B(x_i, R_i)} \frac{1}{|Du|} \, d \Ha^{n-1} \geq \frac{c}{t} R_i^n \geq \frac{c}{t}\, |B(x_i,R_i)| 
\end{equation} 
for every ball $B(x_i, R_i) \subset B_{2r}$ in the cover.  Summing \eqref{low density 1} gives
\[
\begin{split}
 \int_{\{u=t\} \cap B_{2r} }  \frac{1}{|Du|} \, d \Ha^{n-1}&\geq  \sum_i   \int_{\{u=t\}\cap B(x_i, R_i)} \frac{1}{|Du|} \, d \Ha^{n-1}  \\
&\geq \frac{c}{t}\, \sum_i |B(x_i,R_i)| \\
&\geq \frac{c}{\xi(n) t} 	\, |E_t \cap B_r|
\end{split}
\]
and the claim follows.

Let us assume next
\begin{equation} \label{case big}
|E_t \cap  B_r| \geq  2^{-n-2} |B_r|.
\end{equation}
In this case we do not need any covering argument. Since we assumed that   $|E_t \cap  B_r| \leq  |B_r \setminus E_t |$ we have 
\[
|E_t \cap  B_r| \leq \frac{|B_r|}{2}.
\]
We use Proposition \ref{cacccioppoli} in $B_r$  and argue as above to conclude that  
\[
\int_{\{u=t\} \cap B_r} |Du| \, d \Ha^{n-1} \leq    C\left(  \frac{t}{r^2} + \frac{F(2t)}{t} \right) |E_t \cap B_{2r}| \leq  Ct \, r^{n-2},
\]
where we used the fact that $r  \leq  t/\sqrt{F(2t)}$. We use the relative isoperimetric inequality  and \eqref{case big}, argue  as in \eqref{half density 2} and get 
\[
\int_{\{u=t\}\cap B_r} |Du| \, d \Ha^{n-1} \geq c \left( \int_{\{u=t\} \cap B_r} \frac{1}{|Du|} \, d \Ha^{n-1}  \right)^{-1}  r^{2n-2}.
\]
Hence the claim follows. 
\end{proof}

Next we prove a measure theoretical lemma which  is  related to the relative  isoperimetric inequality. 
To this aim we denote  for every $\delta \in (0,1]$  the truncated distance function  to the boundary $\partial B_2$ by
\begin{equation} \label{truncated}
d_\delta(x) :=c \min\{(2- |x|), \delta\},
\end{equation}
where $c \in (0,1]$ is a number which we choose later.  
For every measurable set $A \subset B_2$ we define  a density function $\sigma_A: B_2 \to \R$ by
\begin{equation} \label{density function}
\sigma_A(x):= \frac{|A \cap B(x, d_\delta(x))|}{|B_{d_\delta(x)}|}.
\end{equation}
Although it is not apparent from the notation, the function $\sigma_A$ depends on $\delta$. The point of the following lemma
is to study the size of the set where the density function of a given set $A$ takes values between  $1/5$ and $4/5$, i.e., away from zero and one. 
Heuristically one may think that the set $\{ \sigma_A = 1/2\}$ corresponds to the boundary of $A$ and that  $\{1/5 < \sigma_A < 4/5 \}$ forms a layer around it. 
The thickness of this layer depends on the Lipschitz constant of $\sigma_A$ which can be estimated by a simple geometrical argument as follows. 

Fix $x \in B_2$ and a unit vector $e$. When $h>0$ is  small it holds
\[
B(x+ he, d_\delta(x+ he)) \subset B(x, d_\delta(x)+2h).
\]
Therefore  we may estimate
\[
|A \cap B(x+ he, d_\delta(x+ he))| - |A \cap B(x, d_\delta(x))|\leq  |B(x, d_\delta(x)+2h) \setminus  B(x, d_\delta(x))|\leq C d_\delta(x)^{n-1}h,
\]
where $C$ depends on the dimension. Moreover it is easy to see that  
\[
 \lim_{h\to 0} \frac{1}{h}\left(\frac{1}{|B_{d_\delta(x+ he)}|} - \frac{1}{|B_{d_\delta(x)}|}\right) \leq \frac{C}{d_\delta(x)^{n+1}}.
\]
Hence we have 
\[
\begin{split}
\sigma_A(x+he) - \sigma_A(x)  &= \frac{|A \cap B(x+ he, d_\delta(x+ he))| - |A \cap B(x, d_\delta(x))|}{|B_{d_\delta(x+ he)}|}\\
&\,\,\,\,\,\,\,+ |A \cap B(x, d_\delta(x))| \left( \frac{1}{|B_{d_\delta(x+ he)}|} - \frac{1}{|B_{d_\delta(x)}|}\right)\\
&\leq  \frac{C}{d_\delta(x)} h + o(h).
\end{split}
\]
Therefore $\sigma_A$ is locally Lipschitz continuous and its gradient can be estimated by
\begin{equation} \label{lipcshitz of u_F}
|D \sigma_A(x)| \leq \frac{C}{d_\delta(x)}
\end{equation}
for almost every $x \in B_2$.

\begin{lemma}
\label{isop ineq 1}
Suppose $A \subset B_2$  is a measurable set such that  $|A| \leq \frac{1}{2}|B_2|$ and let $\delta \in (0,1]$. Let  $d_\delta(\cdot)$ be  the truncated distance  \eqref{truncated}  and let   $\sigma_A$
 be the density function  \eqref{density function}. There are  constants $c_1$ and $\delta_1$ which depend on the dimension  
such that for every $\delta \leq \delta_1$ it holds
\[
\big| \{1/5 < \sigma_A \leq 4/5 \} \big| \geq c_1 \min \big{\{} \delta \,  \big| \{\sigma_A > 4/5 \} \big|^{\frac{n-1}{n}}, \big| \{\sigma_A > 4/5 \} \big| \big{\}}.
\]
\end{lemma}

It is not difficult to see that  if $A$ is a smooth set then by letting $\delta \to 0$ the previous lemma reduces to the relative isoperimetric inequality.  

\begin{proof}
Let us fix $\delta \in (0,1]$. Throughout the proof we denote  for every $s \in [0,1]$  the superlevel sets of $\sigma_A$ by
\[
A^s:= \{x \in B_2 \, : \, \sigma_A > s \}.
\]

For every $x \in B_{2-\delta}$ it holds $ d_\delta(x) = c\delta$ and therefore
\[
\sigma_A(x) = \frac{|A \cap B(x, c\delta)|}{|B_{c\delta}|} =  \int_{B_2} \chi_A(y) \chi_{B_{c\delta}}(x-y)\, dy.
\]
By Fubini's theorem we have
\[
\begin{split}
\int_{B_{2-\delta}} \sigma_A(x) \, dx &=  \int_{B_2} \chi_A(y) \left(   \frac{1}{|B_{c\delta}|}  \int_{B_{2-\delta}}    \chi_{B_{c\delta}}(x-y) \, dx \right) dy \\
&\leq  \int_{B_2} \chi_A(y) \left(   \frac{1}{|B_{c\delta}|}  \int_{\R^n}    \chi_{B_{c\delta}}(x-y) \, dx \right) dy = |A|.
\end{split}
\]
Hence it holds
\[
|A| \geq  \int_{B_{2-\delta}} \sigma_A(x) \, dx \geq \int_{A^{3/5}\cap B_{2-\delta}} \sigma_A(x) \, dx \geq \frac{3}{5} |A^{3/5} \cap B_{2-\delta}|.
\]
By taking $\delta $ small we get that $|A^{3/5} | \leq \frac{7}{4}|A|$. Since $|A| \leq \frac{1}{2}|B_2|$ we have
\begin{equation} \label{level setit pienia}
|A^s| \leq \frac{7}{8}|B_2|
\end{equation}
for every $3/5 \leq s \leq 4/5$.

Let us define 
\begin{equation} \label{maaritelma delta n}
\delta_n :=  \frac{\delta}{2n-1} \qquad \text{and} \qquad R_n := 2- \frac{\delta_n}{2}.
\end{equation}
We divide the proof in two cases. Let us first assume that there exists $\hat{s} \in (7/10, 4/5)$ such that  
\begin{equation} \label{paljon sisalla}
|A^{\hat{s}} \cap B_{R_n}| \geq \frac{1}{8}|A^{\hat{s}}|.
\end{equation} 
This means that part of the level set $A^{\hat{s}} = \{ \sigma_A > \hat{s}\}$ is away from the boundary $\partial B_2$. 
We observe that for every $x \in B_{R_n}$ it  holds $d_\delta(x) \geq c \delta$ for a dimensional constant $c>0$. By possibly decreasing $\delta$  we deduce from \eqref{level setit pienia}  that
$|A^{s}\cap B_{R_n}| \leq \frac{8}{9}|B_{R_n}|$ for every  $3/5 \leq s \leq 4/5$.  Then it follows from \eqref{lipcshitz of u_F}, from the coarea formula and from the relative isoperimetric inequality that
\[
\begin{split}
| \{ 1/5 < \sigma_A \leq 4/5\}|  &\geq c \int_{(A^{1/5}\setminus  A^{4/5}) \cap B_{R_n}} d_\delta(x)|D \sigma_A (x)|\, dx\\
&\geq c \delta \int_{3/5}^{\hat{s}} \int_{\{ \sigma_A = s\} \cap B_{R_n}} d \Ha^{n-1}(x) \, ds   \\
&= c \delta \int_{3/5}^{\hat{s}} P(A^{s}, B_{R_n}) \, ds \\
&\geq c \delta \int_{3/5}^{\hat{s}} |A^{s} \cap B_{R_n}|^{\frac{n-1}{n}} \, ds\\
&\geq c \delta  |A^{\hat{s}} \cap B_{R_n}|^{\frac{n-1}{n}}.
\end{split}
\]
Hence the claim follows from \eqref{paljon sisalla} since 
\[
|A^{\hat{s}} \cap B_{R_n}| \geq \frac{1}{8}|A^{\hat{s}}| \geq \frac{1}{8}|A^{4/5}| = \frac{1}{8}|\{ \sigma_A >4/5\}| .
\]

Let us next assume that  for every $s \in   (7/10, 4/5)$ it holds
\begin{equation} \label{paljon reunalla}
|A^s \cap B_{R_n}| \leq \frac{1}{8}|A^s|.
\end{equation} 
This means that  large part of the level set $A^s$ is close to the boundary $\partial B_2$.

We denote the reduced boundary of $A^s$ by $\partial^* A^s$.   Let us first show that for almost every $s \in  (7/10, 4/5)$ it holds
\begin{equation}
\label{lahella reunaa}
\int_{\partial^*  A^s \cap B_{2}} d_\delta(x) \, d\Ha^{n-1} \geq c|A^s|.
\end{equation}

To this aim  fix $s \in  (7/10, 4/5)$ such that $A^s$ has  locally  finite perimeter in $B_2$. We deduce from the definition of $R_n$ \eqref{maaritelma delta n}, from \eqref{paljon reunalla} and from 
 the coarea formula that there is  $\rho_n \in (2- \delta_n,R_n)$ such that 
\begin{equation} \label{fubini 1}
\begin{split}
\frac{1}{8}|A^s| \geq |A^s  \cap B_{R_n}  | &\geq |A^s \cap (B_{R_n}\setminus B_{2-\delta_n})| \\
&= \int_{2-\delta_n}^{R_n} \Ha^{n-1}(\partial B_\rho \cap A^s) d \rho\\
&= \frac{\delta_n}{2} \Ha^{n-1}(\partial B_{\rho_n} \cap A^s).
\end{split}
\end{equation}
Choose a vector field $X(x) = (|x|-2)\frac{x}{|x|}$ and observe that by the definition of $\delta_n$ \eqref{maaritelma delta n} it holds
\[
(\diver X)(x) = n-\frac{2(n-1)}{|x|} \geq \frac{1}{2}
\] 
for every $|x|\geq 2- \delta_n$. Therefore  the Gauss-Green formula  and \eqref{fubini 1}  yield
\[
\begin{split}
\frac{1}{2}|A^s \setminus B_{\rho_n}| &\leq \int_{A^s \setminus B_{\rho_n}} \diver(X) \, dx\\
&=  \int_{\partial^* A^s \cap (B_2 \setminus B_{\rho_n}) } (|x|-2) \la \frac{x}{|x|}, \nu \ra\, \Ha^{n-1} +  \int_{\partial B_{\rho_n} \cap A^s }\, (2- |x|) \Ha^{n-1} \\
&\leq C \int_{\partial^* A^s \cap (B_2 \setminus B_{\rho_n}) } d_\delta(x) \, \Ha^{n-1}  + \delta_n \Ha^{n-1}(\partial B_{\rho_n} \cap A^s)\\
&\leq C \int_{\partial^* A^s \cap B_2 } d_\delta(x)  \, \Ha^{n-1}  + \frac{1}{4}|A^s|.
\end{split}
\]
The inequality \eqref{lahella reunaa} then   follows from \eqref{paljon reunalla} as follows
\[
|A^s \setminus B_{\rho_n}| \geq |A^s \setminus B_{R_n}| \geq \frac{7}{8}|A^s|.
\]

Finally we use \eqref{lipcshitz of u_F}, the coarea formula and  \eqref{lahella reunaa} to conclude 
\[
\begin{split}
| \{ 1/5 < \sigma_A \leq 4/5\}|   &\geq c \int_{A^{7/10}\setminus  A^{4/5}} d_\delta(x)|\nabla \sigma_A(x)|\, dx\\
&=  c \int_{7/10}^{4/5} \left( \int_{ \{ \sigma_A = s \} \cap B_2} d_\delta(x) \, d \Ha^{n-1}(x)\right) ds  \\
&= c \int_{7/10}^{4/5} \left( \int_{\partial^* A^s \cap B_2} d_\delta(x) \, d \Ha^{n-1}(x)\right) ds \\
&\geq c \int_{7/10}^{4/5} |A^s| \, ds\\
&\geq c   |A^{4/5}| .
\end{split}
\]
\end{proof}

We use the two previous  lemmas to estimate  the decay rate of the level sets.  

\begin{proposition}
\label{decay oma}
Assume that $u \in C^{\infty}(B_{2})$ is a nonnegative supersolution of 
\[
\diver(A(x)Du) \leq f(u).
\]
Denote $\mu(t) = |A_t \cap B_{2}|$ and $\eta(t) =  |E_t \cap B_{2}|$ where $E_t = \{ u <t\}$ and $A_t = \{ u\geq t\}$.  Then  for almost every $t>0$ with $\mu(t) \leq \frac{|B_2|}{2}$ it holds
\[
 -\mu'(t) \geq c \min \Bigl\{  \frac{1}{\sqrt{F(2t)}}  \mu(t)^{\frac{n-1}{n}} , \frac{1}{t}\mu(t)\Bigl\}
\]
and for almost every $t>0$ with $\eta(t) \leq \frac{|B_2|}{2}$  it holds
\[
 \eta'(t) \geq c \min \Bigl\{  \frac{1}{\sqrt{
F(2t)}}  \eta(t)^{\frac{n-1}{n}} , \frac{1}{t}\eta(t)\Bigl\}.
\]
Here $\mu'$ is the absolutely continuous part of the differential of $\mu$.
\end{proposition}

\begin{proof}
We will only prove the first inequality since the second follows from  a similar argument. By Morse-Sard Lemma for almost every $t >0$ it holds  that $|Du| >0$ on   $\{u=t\} \cap B_2$ and  by \cite{AL}
we may write
\[
-\mu'(t) = \int_{\{u=t\}\cap B_2} \frac{1}{|Du|}\, d\Ha^{n-1}.
\]
Let us fix $t >0$ for which this holds. Let $\delta_1 >0$ be from Lemma \ref{isop ineq 1}. We choose
\[
\delta = \min\Bigl{ \{ }\frac{t}{\sqrt{F(2t)}}, \delta_1 \Bigl{\}},
\]
 define the truncated distance by 
\[
d_\delta(x):= \frac{1}{100}\min\{(2-|x|), \delta \},
\]
and  denote the density function by
\[
\sigma_{A_t}(x):= \frac{|A_t \cap B(x, d_\delta(x))|}{|B_{d_\delta(x)}|}.
\]

Let us divide the proof in two cases and assume first that  
\[
\big|\{ \sigma_{A_t} > 4/5\}\big| \leq \frac{1}{2}|A_t|.
\]
This means that there is a large set where the density of  $A_t$ is low. In particular,  it holds
\[
\big|A_t \cap \{ \sigma_{A_t} \leq 4/5\} \big| \geq \frac{1}{2}|A_t|
\]
For every $x_i \in  A_t \cap \{ \sigma_{A_t}(x) \leq 4/5\}$ we choose a  ball $B(x_i, d_\delta(x_i))$ and thus  obtain a covering 
of $A_t \cap \{ \sigma_{A_t} \leq 4/5\}$. By  the Besicovitch covering theorem \cite[Corollary 5.2]{Maggi} we may choose a countable disjoint subfamily, say $\mathcal{F}$, such that 
\begin{equation} \label{besi cover 1}
\xi(n)\sum_{i \in \mathcal{F}}\big| \left(A_t \cap \{ \sigma_{A_t} \leq 4/5\}\right) \cap B(x_i, d_\delta(x_i))\big| \geq \big| A_t \cap \{ \sigma_{A_t} \leq 4/5\}\big| \geq  \frac{1}{2}|A_t| .
\end{equation}
We observe that if $B(x_i, d_\delta(x_i))$ and $B(x_j, d_\delta(x_j))$ are two balls from $\mathcal{F}$ such that 
the enlarged balls $B(x_i, 2d_\delta(x_i))$ and  $B(x_j, 2d_\delta(x_j))$ overlap each other, then their radii are comparable,  
\[
\frac{d_\delta(x_i)}{3} \leq d_\delta(x_j) \leq 3 d_\delta(x_i).
\]
Therefore since the balls in  $\mathcal{F}$ are disjoint it follows that the  enlarged balls $B(x_i, 2d_\delta(x_i))$
intersect every point in $B_2$ only  finitely many times, say, $\eta(n)$ times. In particular, we have
\begin{equation} \label{isommat pallot ei leikkaa}
\eta(n) \int_{\{u=t\}\cap B_2} \frac{1}{|Du|}\, d\Ha^{n-1} \geq  \sum_{i \in \mathcal{F}}\int_{\{u=t\} \cap B(x_i, 2d_\delta(x_i))} \frac{1}{|Du|}\, d\Ha^{n-1}.
\end{equation}

Let $B(x_i, d_\delta(x_i)) \in \mathcal{F}$. Since $x_i \in \{ \sigma_{A_t} \leq 4/5\}$ we have 
\[
\big|B(x_i, d_\delta(x_i))\setminus E_t \big|  =  \big|A_t \cap B(x_i, d_\delta(x_i))\big| \leq \frac{4}{5} \big|B(x_i, d_\delta(x_i))\big|. 
\]
Therefore 
\[
\big|E_t\cap B(x_i, d_\delta(x_i))\big| \geq \frac{1}{4} \big|B(x_i, d_\delta(x_i))\setminus E_t\big| 
\]
and  Lemma \ref{caccio invert} gives 
\[
\begin{split}
 \int_{\{u=t\} \cap B(x_i, 2 d_\delta(x_i))}\frac{1}{|Du|}\, \Ha^{n-1} &\geq \frac{c}{t}\, \big|B(x_i,d_\delta(x_i)) \setminus E_t \big| \\
&=  \frac{c}{t}\, \big|A_t \cap B(x_i,d_\delta(x_i))\big| .
\end{split}
\]
Summing the above inequality and using \eqref{besi cover 1} and \eqref{isommat pallot ei leikkaa} yield
\[
\begin{split}
\eta(n) \int_{\{u=t\} \cap B_2} \frac{1}{|Du|}\, d\Ha^{n-1} &\geq  \sum_{i \in \mathcal{F}}\int_{\{u=t\} \cap B(x_i, 2d_\delta(x_i))} \frac{1}{|Du|}\, d\Ha^{n-1}\\
&\geq  \frac{c}{t}  \sum_{i \in \mathcal{F}}  \big|A_t \cap B(x_i,d_\delta(x_i))\big| \\
&\geq \frac{c}{\xi(n)t} |A_t|.
\end{split}
\]
In other words $-\mu'(t) \geq \frac{c}{t}\mu(t)$ and the claim follows in this case.

Let us next assume that
\begin{equation} \label{decay lemma oletus 2}
\big|\{ \sigma_{A_t} > 4/5 \}\big| \geq \frac{1}{2}|A_t|.
\end{equation}
This means that there is a large set where the density of $A_t$ is close to one. In this case we  do not cover the  set $\{ \sigma_{A_t} \leq 4/5 \}$ but only the part where the density $\sigma_{A_t}$ is between $1/5$ and $4/5$.
Then we estimate the measure of this set using Lemma \ref{isop ineq 1} .

For every $x_i \in  \{1/5 <  \sigma_{A_t} \leq  4/5\}$ we choose a  ball $B(x_i, d_\delta(x_i))$ and thus  obtain  a covering 
of $\{1/5 <  \sigma_{A_t} \leq  4/5\}$. By  the Besicovitch covering theorem we may choose a countable disjoint subfamily  $\mathcal{F}$ such that 
\begin{equation} \label{besi cover 2}
\big|\{1/5 <  \sigma_{A_t} \leq  4/5\}\big| \leq \xi(n)\sum_{i \in \mathcal{F}}\big|B(x_i, d_\delta(x_i))\big|.
\end{equation}
Moreover as in \eqref{isommat pallot ei leikkaa} we have 
\begin{equation} \label{isommat pallot ei leikkaa 2}
\eta(n) \int_{\{u=t\} \cap B_2} \frac{1}{|Du|}\, d\Ha^{n-1} \geq  \sum_{i \in \mathcal{F}}\int_{\{u=t\} \cap B(x_i, 2d_\delta(x_i))} \frac{1}{|Du|}\, d\Ha^{n-1}.
\end{equation}

Let $B(x_i, d_\delta(x_i)) \in \mathcal{F}$. Since $x_i \in \{1/5 <  \sigma_{A_t} \leq  4/5\}$ we have 
\[
 \frac{1}{5}\leq \frac{\big|B(x_i, d_\delta(x_i))\setminus E_t \big|}{\big|B_{d_\delta(x_i)}\big|}  = \frac{\big|A_t \cap B(x_i, d_\delta(x_i))\big|}{\big|B_{d_\delta(x_i)}\big|}  \leq \frac{4}{5}. 
\]
Therefore Lemma  \ref{caccio invert} gives  
\[
 \int_{\{u=t\} \cap B(x_i, 2 d_\delta(x_i))}\frac{1}{|Du|}\, \Ha^{n-1} \geq \frac{c}{t}\,\big|B(x_i,d_\delta(x_i))\big|.
\]
Summing the above inequality and using  \eqref{besi cover 2} and \eqref{isommat pallot ei leikkaa 2} yield
\[
\begin{split}
\eta(n) \int_{\{u=t\} \cap B_2} \frac{1}{|Du|}\, d\Ha^{n-1} &\geq  \sum_{i \in \mathcal{F}}\int_{\{u=t\} \cap B(x_i, 2d_\delta(x_i))} \frac{1}{|Du|}\, d\Ha^{n-1}\\
&\geq  \frac{c}{t}  \sum_{i \in \mathcal{F}} \big|B(x_i,d_\delta(x_i))\big| \\
&\geq \frac{c}{\xi(n) t} \big|\{1/5 <  \sigma_{A_t} \leq  4/5\}\big|.
\end{split}
\]
Finally we use Lemma \ref{isop ineq 1} to conclude that 
\[
\big|\{1/5 <  \sigma_{A_t} \leq  4/5\}\big| \geq  c \min \big{\{} \delta \,  \big|\{\sigma_{A_t} >  4/5\}\big|^{\frac{n-1}{n}}, \big|\{\sigma_{A_t} >  4/5\}\big| \big{\}}.
\]
Therefore  by  the two previous inequalities and by \eqref{decay lemma oletus 2} we get
\[
- \mu'(t) = \int_{\{u=t\} \cap B_2} \frac{1}{|Du|}\, d\Ha^{n-1}  \geq \frac{c}{t}\min \big{\{} \delta \,  \mu(t)^{\frac{n-1}{n}}, \mu(t)\big{\}}.
\]
The result follows from 
\[
\delta = \min\Bigl{ \{ }\frac{t}{\sqrt{F(2t)}}, \delta_1 \Bigl{\}}.
\]
\end{proof}

\section{Estimates for subsolutions}

In this section we prove estimates for subsolutions of the equation \eqref{the pde}. The most important result of this section for Theorem \ref{thm1}  is  Lemma \ref{de giorgi subille}.
 Similar but slightly different result  is proved  in \cite{Konkov2}.  The proof is fairly standard. Instead of
 using a capacity argument \cite{Konkov2} we give a short proof based on  De Giorgi iteration. After Lemma \ref{de giorgi subille} we give the proof of Theorem \ref{theorem sub}
 and  Corollary \ref{keller osserman oma}. 

We begin by recalling the following  standard result  for subsolutions of the linear equation
\begin{equation} \label{sub linear}
\diver(A(x)Du) = 0.
\end{equation}

\begin{lemma}
\label{perus sub}
Let $u \in W^{1,2}(B(x_0,2r))$ be a subsolution of  \eqref{sub linear} and denote  $M_r := \sup_{B(x_0, r)}u$. There is 
$\eps_0 >0$ such that if 
\[
\big|\{ u \geq  u(x_0)/2\} \cap B(x_0,  r) \big| \leq \eps_0 |B_r|,
\]
then $M_r \geq 8 u(x_0)$.
\end{lemma}

\begin{proof}
Let us recall that by De Giorgi theorem any subsolution $v$ of  \eqref{sub linear}
is locally bounded and satisfies 
\[
\sup_{B(x_0, r/2)} v \leq C r^{-\frac{n}{2}}||v||_{L^2(B(x_0,  r))}.
\]
The result follows by applying this to $v = (u- u(x_0) /2)_+$ which is a subsolution of \eqref{sub linear}.
\end{proof}

We recall the notation  $A_t = \{ u \geq t\}$.  We have the following Caccioppoli inequality for subsolutions of \eqref{the pde}.

\begin{lemma}
\label{caccio subille}
Let $u \in W^{1,2}(B(x_0,2r))$ be a subsolution of
\[
\diver(A(x)Du) \geq f(u)
\]
and denote  $ M_r:= \sup_{B(x_0, r)}u$. Then for every  $\rho<r$ and $t <M_r$  it holds
\[
 \int_{B(x_0, \rho)} (u-t)_+^2\, dx  \leq \frac{C}{(r-\rho)^2} \, |A_t \cap B_r|^{\frac{1}{n}} \left( \frac{M_r}{\sqrt{F(t)}} \right) \int_{B(x_0, r)} (u-t)_+^{2} \, dx. 
\]
\end{lemma}

\begin{proof}
Without loss of generality we may assume that $x_0 = 0$. We use testfunction $\varphi = (u-t)_+ \zeta^2$ in \eqref{the pde}  
where $\zeta \in C_0^\infty(B_r)$ is a standard cut-off function such that $\zeta =1$ in $B_\rho$ and $|D\zeta|\leq \frac{2}{r-\rho}$.
This gives
\[
\int_{B_r}|D((u-t)_+ \zeta) |^2 \, dx \leq C \int_{B_r}(u-t)_+^2|D\zeta|^2\, dx - \int_{B_r}f(u)(u-t)_+\zeta^2\, dx. 
\]
Since $f$ is nondecreasing we have for every  $x \in A_t \cap B_r$ that 
\[
f(u(x)) \geq f(t) \geq \frac{F(t)}{t} \geq \frac{F(t)}{M_r^2} (u(x)-t). 
\]
Therefore we have 
\[
 \int_{B_r}f(u)(u-t)_+\zeta^2\, dx \geq   \frac{F(t)}{M_r^2} \int_{B_r}(u-t)_+^2\zeta^2\, dx.
\] 
We denote $w = (u-t)_+ \zeta \in W_0^{1,2}( B_r)$  and obtain  by Young's and Sobolev inequalities 
\[
\begin{split}
\int_{B_r}|D((u-t)_+ \zeta) |^2 \, dx &+ \frac{F(t)}{M_r^2} \int_{B_r}(u-t)_+^2\zeta^2\, dx \\
&\geq c\frac{\sqrt{F(t)}}{M_r}  \int_{B_r}| w Dw|\, dx \\
&= c \frac{\sqrt{F(t)}}{M_r} \int_{B_r} \frac{1}{2}|  D (w^2)|\, dx \\
&\geq c \frac{\sqrt{F(t)}}{M_r} \left( \int_{B_r} w^{\frac{2n}{n-1}}\, dx \right)^{\frac{n-1}{n}}.
\end{split}
\]

The result follows from Jensen's inequality 
\[
\left( \int_{B_\rho} (u-t)_+^{\frac{2n}{n-1}}\, dx \right)^{\frac{n-1}{n}} \geq  |A_t \cap B_\rho|^{-\frac{1}{n}} \int_{B_\rho} (u-t)_+^2\, dx.
\]
\end{proof}

By  the standard De Giorgi iteration we obtain the estimate we need. 

\begin{lemma}
\label{de giorgi subille}
Let $u \in W^{1,2}(B(x_0,2r))$ be a   subsolution of
\[
\diver(A(x)Du) \geq f(u)
\]
and denote $M_{r}:= \sup_{B(x_0,r)}u$.  If $0 < M_{r} \leq 2 u(x_0)$,  then  
\[
\frac{ u(x_0)}{\sqrt{F\left(\frac{u(x_0)}{2}\right)}} \geq  c r.
\] 
\end{lemma}

\begin{proof}
By rescaling and translating we may assume that  $u$ is a subsolution of
\[
\diver(A(x)Du) \geq r^2f(u)
\]
in $B_2$ and $0 < M_1 \leq 2 u(0)$.  We need to show that
\[
\Lambda_0 := \frac{u(0)}{\sqrt{r^2 F\left(\frac{u(0)}{2}\right)}}\geq c >0.
\]

 Let $\tau, \rho$ be such that $1/2 < \rho < \tau < 1$ and $h,k$ such that $u(0)/2 \leq h<k \leq u(0)$.
 Then it follows from the assumption $M_{1} \leq 2 u(0)$ and Lemma \ref{caccio subille} that
\[
\int_{A_{k} \cap B_{\rho}} (u-k)_+^2 \, dx \leq C \Lambda_0  \, |A_{k} \cap B_{\tau}|^{\frac{1}{n}} \frac{1}{(\tau-\rho)^2}  \int_{A_{h} \cap B_{\tau}} (u-h)_+^2 \, dx.
\]
Since 
\[
 |A_{k} \cap B_{\tau}|  \leq \frac{1}{(k-h)^2} \int_{A_{h} \cap B_{\tau}} (u-h)_+^2 \, dx
\]
one gets
\begin{equation} \label{giusti 7.19}
\int_{A_{k} \cap B_{\rho}} (u-k)_+^2 \, dx \leq C\Lambda_0  \, \frac{1}{(k-h)^\frac{2}{n}}  \frac{1}{(\tau-\rho)^2}  \left( \int_{A_{h} \cap B_{\tau}} (u-h)_+^2 \, dx \right)^{1+\frac{1}{n}}.
\end{equation}

Define $r_i = \frac{1}{2}(1+ 2^{-i})$ and $k_i = u(0)\left( \frac{3}{4}-4^{-1-i}\right)$. Write \eqref{giusti 7.19} for $h = k_i$, $k = k_{i+1}$, $\tau = r_i$ and $\rho = r_{i+1}$ and set
\[
x_i = M_1^{-2} \int_{A_{k_i} \cap B_{r_i}} (u-k_i)_+^2 \, dx. 
\]
Since $M_1 \leq 2u(0)$ we obtain 
\[
x_{i+1} \leq C_0 C_1^i  \Lambda_0 \,  x_i^{1+ \frac{1}{n}}
\]
for $C_0 >1$ and $C_1>1$. Let us show that it holds
\begin{equation} \label{contradiction de giorgi}
\Lambda_0  \geq |B_1|^{-\frac{1}{n}} C_0^{-1}  C_1^{-n}.
\end{equation}
Indeed, we argue by contradiction and assume that \eqref{contradiction de giorgi} does not hold. This implies 
\[
x_0 \leq  M_1^{-2} \int_{B_{1}} u^2 \, dx \leq |B_1| \leq  C_0^{-n}C_1^{-n^2} \Lambda_0^{-n}.
\]
It follows from Lemma  \ref{algebra lemma} that $\lim_{i \to \infty} x_i = 0$. This means that 
\[
\sup_{B_{1/2}} u \leq \frac{3}{4}u(0)
\]
which is a contradiction. Therefore we have  \eqref{contradiction de giorgi} and the claim follows.
\end{proof}

\begin{proof}[\textbf{Proof of Theorem \ref{theorem sub}}]
Without loss of generality we may assume that $x_0 = 0$. Let us denote $M_0 =  u(0)$ and $R_0 = 0$. We define radii $R_k \in (0,R]$, where $k =1, \dots, K$, such that the corresponding maxima $M_k := \sup_{B_{R_k}}u$ satisfy
\[
M_{k} = 2M_{k-1}
\]
for every $k =1, \dots, K-1$ and $R_K = R$ and  $\sup_{B(x_0,R)} u  = M_{K} \leq 2M_{K-1}$. Denote also   $r_k = R_{k}- R_{k-1}$ for $i =1, \dots, K$. For notational reasons we also define
$M_{-1}= M_0/2$ and  $M_{-2}= M_0/4$.

Let us fix $k  \in 1, \dots, K$.  Let  $x_{k-1} \in \partial B_{R_{k-1}}$  be a point such that $u(x_{k-1}) =M_{k-1}$. Note that because of the maximum principle the maximum is attained 
on the boundary. Since $\sup_{B(x_{k-1}, r_k)} u \leq M_{k} =  2M_{k-1}$  Lemma \ref{de giorgi subille} yields
\[
\frac{M_{k-1}}{\sqrt{F(M_{k-1}/2)}} \geq  c r_k.
\]
Note that $M_{k-1} =  4M_{k-3}$ and   $ M_{k-1}/2 =M_{k-2}$. Therefore by the monotonicity of $F$  it holds  
\[
\int_{M_{k-3}}^{M_{k-2}}\frac{t}{\sqrt{F(t)}} \geq \frac{M_{k-3}}{\sqrt{F(M_{k-2})}} \geq \frac{1}{4}  \frac{M_{k-1}}{\sqrt{F(M_{k-1}/2)}}  \geq  c r_k.
\]
Summing the above inequality over $k =  1, \dots, K$  yields
\[
\int_{M_{-2}}^{M_{K-2}} \frac{dt}{\sqrt{F(t)}} \geq c \sum_{k=1}^{K}r_k =c \, R. 
\]
Since $M_{-2} = u(0)/4 \geq m/4$  and $M_{K-2} \leq M$ one has 
\[
\int_{M_{-2}}^{M_{K-2}} \frac{dt}{\sqrt{F(t)}} \leq \int_{m/4}^{M} \frac{dt}{\sqrt{F(t)}}. 
\]
The result follows from the previous two inequalities. 
\end{proof}
 
Let us briefly discuss about the previous theorem. At the  first glance the statement of  Theorem \ref{theorem sub} might seem unsatisfactory. 
First, one could think that the assumption $u(x_0)>0$ in Theorem \ref{theorem sub} is unnecessary. However, it is easy to see that it can not be removed. 
Second, one can not reduce the interval of integration  from $[m/4,M]$ to $[m,M]$, i.e., the estimate
\begin{equation} \label{vaara estimaatti}
\int_{m}^M \frac{dt}{\sqrt{F(t)}} \geq c \, R
\end{equation}
is not true.  To see this  choose $f$ such that $f(t)=0$ for $t \in [0,1]$ and $f(t)= 1$ for $t>1$. Construct a one dimensional solution of $u'' = f(u)$ in $(0,1)$ by
$u(x) = 1$ for $0 < x \leq 1- \eps$ and $u(x)= \frac{1}{2}(x -1 + \eps)^2+1$ for $1- \eps < x <1$. This solution does not satisfy the estimate \eqref{vaara estimaatti}.

\begin{corollary}
\label{keller osserman oma}
Suppose that there exists   a continuous subsolution   $u \in W_{loc}^{1,2}(\R^n)$ of 
\[
\diver(A(x)Du) \geq f(u)
\]
in $\R^n$ which is not constant. Then necessarily
\[
\int_{-\infty}^\infty \frac{dt}{\sqrt{F(t)}} = \infty.
\] 
\end{corollary} 

\begin{proof}
 Since $u$ is not constant there exists a point $x_0$ such that $u(x_0)> \inf_{B(x_0,1)}u$. Without loss of generality we may assume that $x_0 = 0$.  Let us fix a large radius $R>1$ and denote $m_{2R}= \inf_{B_{2R}} u$ and $M_R = \sup_{B_R} u$. We define $v = u-m_{2R}$ which is a nonnegative subsolution of 
\[
\diver(A(x)Dv) \geq \tilde{f}(v)
\]
in $B_{2R}$ where $\tilde{f}(t) = f(t+m_{2R})$.  Denote also $\tilde{M} =  \sup_{B_R}v = M_R -m_{2R}$ and $\tilde{m} =  \inf_{B_R}v =m_R- m_{2R}$. 
Since $v(0)= u(0)-m_{2R}>0$ we deduce from Theorem \ref{theorem sub} that  
\[
\int_{\tilde{m}/4}^{\tilde{M}} \frac{dt}{\sqrt{\tilde{F}(t)}} \geq c\,  R.
\]
Since 
\[
\int_{\tilde{m}/4}^{\tilde{M}} \frac{dt}{\sqrt{\tilde{F}(t)}} \leq \int_{m_{2R}}^{M_R} \frac{dt}{\sqrt{F(t)}} 
\]
the result follows by letting $R \to \infty$.
\end{proof}

\section{Proof of the Main Theorem}

This section is devoted to the proof of Theorem \ref{thm1}.

\begin{proof}[\textbf{Proof of Theorem \ref{thm1}}]
Let   $u$ be a nonnegative solution of 
\[
\diver(A(x)Du) = f(u)
\]
in $B_2$. As we discussed in Section 2 we may assume that $u \in C^{\infty}(B_2)$. Let us once more recall our  notation  $E_t = \{u < t\} $, $A_t =  \{ u \geq t\}$, $\eta(t) = |E_t \cap B_2|$ and $\mu(t) = |A_t \cap B_2|$. Moreover $m = \inf_{B_1}u $ and $M  = \sup_{B_1}u$. We first note that it is enough to show that 
\begin{equation} \label{naytettava}
\int_m^{M} \frac{dt}{\sqrt{F(2t)} +t} \leq C. 
\end{equation}
Indeed, then the claim follows by change of variables, as follows   
\[
\begin{split}
\int_m^{M} \frac{ds}{\sqrt{F(s)} +s} &\leq  2  \int_{m/2}^{M/2} \frac{dt}{\sqrt{F(2t)} +t}\\
&\leq   2   \int_{m/2}^{m} \frac{dt}{t} + 2 \int_m^{M} \frac{dt}{\sqrt{F(2t)} +t}\\
&\leq 2 \log 2 + C.
\end{split}
\]

Let $t_0>0$ be such that $\eta(t_0)= \frac{|B_2|}{2}$ or, if such number does not exist, the supremum over the  numbers $t$ for which  $\eta(t) \leq \frac{|B_2|}{2}$.  Let $\delta_0>0$ be from Lemma \ref{DT1}. Let us first show that there exists $C$, which depends on $\delta_0$, such that 
either the integral is bounded 
\[
\int_0^{t_0} \frac{ds}{\sqrt{F(2s)} +s} \leq C
\]
or there exists $t_\delta \in (0,t_0)$ such that 
\begin{equation} \label{raja alhaalta 1}
\eta(t_\delta)\leq \delta_0 \left( \frac{t_\delta}{\sqrt{F(2t_\delta)} +t_\delta} \right)^n \qquad \text{and} \qquad \int_{t_\delta}^{t_0}  \frac{ds}{\sqrt{F(2s)} +s} \leq C.
\end{equation}
Indeed, it follows from   Proposition \ref{decay oma} that for every $t \leq t_0$ for which the first inequality in \eqref{raja alhaalta 1} does not hold, i.e., 
\[
\eta(t)  \geq \delta_0 \left( \frac{t}{\sqrt{F(2t)} +t} \right)^n
\]
we have
\[
\eta'(t) \geq c\frac{ \eta^{1- \frac{1}{n}}(t) }{\sqrt{F(2t)} +t}
\]
for a constant which depends on $\delta_0$. In other words
\[
\frac{d}{dt} (\eta^{\frac{1}{n}}(t)) \geq \frac{c}{\sqrt{F(2t)} +t}.
\]
We  integrate the above inequality over $(t,t_0)$ and get
\[
|B_2| \geq \eta^{\frac{1}{n}}(t_0) - \eta^{\frac{1}{n}}(t)  \geq c \int_t^{t_0} \frac{ds}{\sqrt{F(2s)} +s}. 
\]
Therefore if
\[
\int_0^{t_0} \frac{ds}{\sqrt{F(2s)} +s} \geq \frac{2|B_2|}{c}
\]
we conclude that there exists $t_\delta \in (0,t_0)$ for which  \eqref{raja alhaalta 1} holds. Since $t_\delta$ satisfies \eqref{raja alhaalta 1} we deduce from  Lemma \ref{DT1} that 
\[
m \geq \frac{t_\delta}{2}.
\]
Therefore we have
\begin{equation}
\label{minimin raja}
\int_m^{t_0} \frac{ds}{\sqrt{F(2s)} +s} \leq  \int_{t_\delta/2}^{t_\delta} \frac{ds}{s} +  \int_{t_\delta}^{t_0} \frac{ds}{\sqrt{F(2s)} +s} \leq C
\end{equation}
for a constant $C$ which is independent of $u$ and $f$.

To prove \eqref{naytettava} we need yet to show that 
\[
\int_{t_0}^M \frac{ds}{\sqrt{F(2s)} +s} \leq C.
\]
To this aim let $\eps >0$ be a small number which we choose later. Let $t_\eps \geq t_0$ be the first value of $u$  such that 
\[
|\{ u \geq t_\eps \} \cap B_2| \leq \eps,
\] 
or, to be  more precise,
\[
t_\eps := \sup \{t >0 : \mu(t) > \eps \}.
\]
Proposition \ref{decay oma} implies that for every $t \in (t_0, t_\eps)$ it holds
\[
- \mu'(t) \geq c\frac{ \mu^{1- \frac{1}{n}}(t) }{\sqrt{F(2t)} +t}
\]
for a constant $c$ which depends on $\eps$.  In other words
\[
-\frac{d}{dt} (\mu^{\frac{1}{n}}(t)) \geq \frac{c}{\sqrt{F(2t)} +t}.
\]
Therefore we have 
\[
|B_2| \geq \mu^{\frac{1}{n}}(t_0) - \mu^{\frac{1}{n}}(t_\eps) \geq c \int_{t_0}^{t_\eps} \frac{dt}{\sqrt{F(2t)} +t}. 
\]
We will show that when $\eps>0$ is chosen small enough  it holds $M \leq 2 t_\eps$. This will imply 
\[
\int_{t_0}^M \frac{dt}{\sqrt{F(2t)} +t} \leq   \int_{t_0}^{t_\eps} \frac{dt}{\sqrt{F(2t)} +t} +   \int_{t_\eps}^{2t_\eps} \frac{dt}{t} \leq C.
\] 
The result then follows from the previous estimate and from \eqref{minimin raja}.

To prove $M  \leq 2 t_\eps$ we argue by contradiction and assume that $M > 2 t_\eps$. Recall that by the assumption on $t_\eps$  it holds $|\{ u \geq M/2 \} \cap B_2| \leq \eps$. We note that by the maximum principle the maximum of 
$u$ in any ball $\bar{B}_R$ is obtained on the boundary. Therefore  when $\eps$ is small  we conclude from   Lemma \ref{perus sub} that 
\begin{equation} \label{paljon steppeja}
  \sup_{B_{\frac{5}{4}}}  u \geq 8 M \qquad \text{and }\qquad \sup_{B_{\frac{7}{4}}} u \geq 8 \sup_{B_{\frac{3}{2}}} u. 
\end{equation}
To prove \eqref{paljon steppeja} choose $x_0 \in \partial B_1$ such that $u(x_0)= M$ and $y_0 \in \partial B_{\frac{3}{2}}$ such that $u(y_0)=  \sup_{B_{\frac{3}{2}}} u$ and 
apply Lemma \ref{perus sub}  in $B(x_0, 1/4)$ and in $B(y_0,1/4)$.

We choose radii $1=R_0 < R_1 < R_2 < \dots$ such that the corresponding maxima 
\[
M_k := \sup_{B_{R_k}}u
\]
satisfy 
\[
 M_k = 2M_{k-1}.
\]
Here we use notation $M_0$ for $M= \sup_{B_1}u$. We continue to do this until the first time we find $R_K$ such that $R_K \geq 3/2$. It follows from  \eqref{paljon steppeja}
that $K \geq 4$ and that $R_K \leq 7/4<2$. In particular, $R_K$ is well defined.  Denote $r_k = R_k- R_{k-1}$. We also deduce  from  \eqref{paljon steppeja} that $R_3 \leq 5/4$ and therefore
\begin{equation}
\label{r_1 ei ole iso}
r_1 +r_2 +r_3= R_3- R_0\leq \frac{1}{4}.
\end{equation}

Let us fix $k = 4,   \dots, K$. I claim that for   small  $c$   it holds
\begin{equation}
\label{r_k ei ole iso}
 \mu^{\frac{1}{n}}(t) \geq c r_k 
\end{equation}
for every $t \in (M_{k-4}, M_{k-2}]$. Note that it is enough to show \eqref{r_k ei ole iso} for $t = M_{k-2}$. We argue by contradiction and assume that 
$\mu^{\frac{1}{n}}(M_{k-2}) \leq c r_k$. This can be written as 
\[
\mu(M_{k-2}) \leq \eps_1| B_{r_{k}}|
\]
for  small $\eps_1 >0$. Since $M_{k-1} = 2  M_{k-2}$ this can be again written as
\[
 |\{ u \geq M_{k-1}/ 2  \} \cap B_2| \leq \eps_1| B_{r_{k}}|.
\]
Let $x_{k-1} \in \partial B_{R_{k-1}}$ be such that $u(x_{k-1}) = M_{k-1}$. When $\eps_1$ is small 
it follows from Lemma \ref{perus sub}  that 
\[
 \sup_{B(x_{k-1}, r_k)} u \geq 8 M_{k-1}.
\]
However since $M_k \geq \sup_{B(x_{k-1}, r_k)} u$  this contradicts the fact that $M_k$ was chosen such that $M_k = 2 M_{k-1}$. Hence we have \eqref{r_k ei ole iso}. 
Note also  that  it follows from 
\[
0< \sup_{B(x_{k-1}, r_k)} u \leq M_k = 2M_{k-1} = 2 u(x_{k-1})
\]
and from Lemma \ref{de giorgi subille} that 
\[
\frac{M_{k-1}}{\sqrt{F(M_{k-1}/2)}} \geq c r_k.
\]
For every $t \in (M_{k-4}, M_{k-3}]$ it holds  $M_{k-1} = 8 M_{k-4} < 8 t$ and $M_{k-1}/2 = 2M_{k-3} \geq  2t$.
Therefore by the monotonicity of $F$ we have 
\begin{equation}
\label{epa homo esti}
\frac{t}{\sqrt{F(2t)}} \geq c r_k
\end{equation}
for every $t \in (M_{k-4}, M_{k-3}]$.

Let us fix $k = 4,\dots, K$. We deduce  from Proposition  \ref{decay oma} and from \eqref{r_k ei ole iso} and \eqref{epa homo esti} that for every
 $t \in (M_{k-4}, M_{k-3}]$ we have the estimate
\[
-\mu'(t) \geq \frac{c}{t} \mu^{1- \frac{1}{n}}(t) r_k. 
\]
In other words
\[
-\frac{d}{dt} (\mu^{\frac{1}{n}}(t)) \geq  \frac{c}{t} r_k.
\]
We integrate this over $(M_{k-4}, M_{k-3})$ and use the fact that $M_{k-3} =2 M_{k-4}$ to  conclude
\[
\mu^{\frac{1}{n}}(M_{k-4})- \mu^{\frac{1}{n}}(M_{k-3}) \geq c r_k.
\]
Recall that $M_0 = M= \sup_{B_1}u$. Summing the previous  inequality over $k = 4,\dots, K$ gives 
\begin{equation}
\label{siisti estimaatti}
\mu^{\frac{1}{n}}(M) \geq \sum_{k=4}^{K}\mu^{\frac{1}{n}}(M_{k-4})- \mu^{\frac{1}{n}}(M_{k-3})  \geq c  \sum_{k=4}^{K}r_k.
\end{equation}
Recall that $r_k = R_{k}-R_{k-1}$, $R_0 =1$ and $R_K \geq 3/2$. Therefore $\sum_{k=1}^K r_k = R_K -R_0 \geq \frac{1}{2}$.
Thus we deduce from  \eqref{r_1 ei ole iso} that  
\[
 \sum_{k=4}^{K}r_k \geq \frac{1}{4}.
\]
Therefore  \eqref{siisti estimaatti} yields 
\[
\mu^{\frac{1}{n}}(M) \geq c.
\]
However, by the choice of $t_\eps$  we have 
\[
\mu(M) = |\{ u \geq M \} \cap B_2| \leq |\{ u \geq  t_\eps \} \cap B_2| \leq \eps
\]
which  is a contradiction when $\eps>0$ is small. 
\end{proof}

At the end let us  show why Theorem \ref{thm1} does not hold  for equation
\begin{equation} \label{pde 2}
-\Delta u = f(u)
\end{equation} 
when $n \geq 2$. If  Thereom \ref{thm1}  would be true for nonnegative solutions of \eqref{pde 2} then there would be $C_0$ such that 
\begin{equation} \label{RR loppu}
\int_m^M \frac{1}{\sqrt{F(t)}+ t } \leq C_0.
\end{equation}
Let $\vphi \geq 0$ be the  fundamental solution of the Laplace equation with $\inf_{B_1}\vphi = 1$ and singularity at the origin. Since $\vphi$ is  unbounded we find a radius $r>0$ such that  for the value  of $\vphi$ on $\partial B_r$, denote it by $T$, it holds
\begin{equation} \label{RR loppu2}
\int_1^{T} \frac{1}{t} \geq 2C_0.
\end{equation}
We define $u \in C^{1,1}(B_2)$ such that $u(x)= \vphi(x)$ for $x \in B_2 \setminus \bar{B}_r$ and $u(x) = a (  r^2-  |x|^2) + T$ for $x \in B_r$. Here $a>0$ is chosen such that $u \in C^{1,1}(B_2)$. Then $u$
is a solution of \eqref{pde 2} for some $f$ which satisfies $f(t) = 0$ for $t \in (0,T)$. Therefore \eqref{RR loppu} can not hold because of \eqref{RR loppu2}.

\section*{Acknowledgment}
 This  work was supported by the  Academy of Finland grant 268393.


\begin{thebibliography}{100}

\bibitem{AL}
\textsc{F. Almgren \& E.H.  Lieb}. 
\emph{Symmetric decreasing rearrangement is sometimes continuous.}  J. Amer. Math. Soc.  \textbf{2}  (1989), 683--773.

\bibitem{AFP}
\textsc{L. Ambrosio, N. Fusco \& D. Pallara}.
\emph{Functions of bounded variation and free discontinuity problems},
in the Oxford Mathematical Monographs.
The Clarendon Press Oxford University Press, New York (2000).

\bibitem{BBC}
\textsc{P. Benilan, H. Brezis \& M. Crandall}. 
\emph{A semilinear equation in $L^1(\R^n)$}.  
Ann. Scuola Norm. Sup. Pisa \textbf{4} (1975), 523--555.



\bibitem{cabre}
\textsc{X. Cabre}.
\emph{Regularity of minimizers of semilinear elliptic problems up to dimension 4.}
Comm. Pure Appl. Math. \textbf{63} (2010), 1362--1380. 


\bibitem{CDLV1}
\textsc{I. Capuzzo Dolcetta, F. Leoni \& A. Vitolo}.
\emph{Entire subsolutions of fully nonlinear degenerate elliptic equations.} Bull. Inst. Math. Acad. Sin. (N.S.)  \textbf{9}  (2014),   147--161. 

\bibitem{CDLV2}
\textsc{I. Capuzzo Dolcetta, F. Leoni \& A. Vitolo}.
\emph{On the inequality $F(x,D^2u)\geq f(u) + g(u)|Du|^q$} . to appear in Math. Ann. 

\bibitem{CM}
\textsc{A. Cianchi \& V. Mazya}.
\emph{Global Lipschitz regularity for a class of quasilinear elliptic equations.}
 Comm. Partial Differential Equations \textbf{36} (2011), 100-133.

\bibitem{degiorgi}
\textsc{E. De Giorgi}.
\emph{Sulla differenziabilita e l'analiticita delle estrmali degli inttegrali multipli regolari.}
Mem. Accad. Sci. Torino (Classi di Sci. mat., fis. e nat.) \textbf{3} (1957), 25--43.

\bibitem{diaz}
\textsc{J.I. Diaz}. 
\emph{Nonlinear Partial Differential Equations and Free Boundaries}. Research Notes in Math., Vol.106 (1985).


\bibitem{DT}
\textsc{E. Di Benedetto \& N. Trudinger}.
\emph{Harnack inequality for quasi-minima of variational integrals.} 
Ann. Inst. H. Poincare (Analyse non-lineaire). \textbf{1} (1984), 295--308.

\bibitem{EG}
\textsc{L.C. Evans  \& R.F. Gariepy}.
\emph{Measure Theory and Fine Properties of Functions.} 
Studies in Advanced Mathematics. CRC Press, Boca Raton, FL,  (1992).



\bibitem{FQS}
\textsc{P. Felmer, A. Quaas \& B. Sirakov},
\emph{Solvability of nonlinear elliptic equations with gradient terms.}
 J. Differential Equations \textbf{254}  (2013),  4327--4346. 


\bibitem{Giusti}
\textsc{E. Giusti}.
\emph{Direct methods in the calculus of variations.} World Scientific Publishing Co., Inc., River Edge, NJ, (2003).  

\bibitem{Vesku} 
\textsc{V. Julin}.
\emph{Generalized Harnack inequality for nonhomogeneous elliptic equations.}
Arch. Ration. Mech. Anal. \textbf{216} (2015), 673--702.
	

\bibitem{Keller}
\textsc{J.B.  Keller}.
\emph{On solutions of $\Delta u \geq f(u)$}.  
Comm. Pure Appl. Math.  \textbf{10}  (1957), 503--510. 


\bibitem{Konkov2}
\textsc{A.A. Kon'kov}. 
\emph{On comparison theorems for elliptic inequalities}. J. Math. Anal. Appl.  \textbf{388}  (2012), 102--124.


\bibitem{Konkov}
\textsc{A.A. Kon'kov}. 
\emph{On solutions of quasilinear elliptic inequalities containing terms with lower-order derivatives}. Nonlinear Anal.  \textbf{90}  (2013), 121--134. 


\bibitem{LU}
\textsc{O.A. Ladyzhenskaia \& N.N. Ural'tzeva}.
\emph{On the smoothness of weak solutions of quasilinear equations in several variables and of variational problems.} 
Comm. Pure Appl. Math. \textbf{14} (1961), 481--495.



\bibitem{Maggi}
\textsc{F. Maggi}.
\emph{Sets of finite perimeter and geometric variational problems. An introduction to geometric measure theory}. 
Cambridge Studies in Advanced Mathematics, 135. Cambridge University Press, Cambridge (2012).

\bibitem{mazya}
\textsc{V. Mazya}.
\emph{On weak solutions of the Dirichlet and Neumann problems.} Trusdy Moskov. Mat.  Obsc. \textbf{20} (1969), 137--172. (In Russian). 
English translation in Trans. Moscow Math. Soc.   \textbf{20} (1969), 137--172.

\bibitem{Moser} 
\textsc{J. Moser}.
\emph{On Harnack's theorem for elliptic differential equations.} 
Comm. Pure Appl. Math. \textbf{14} (1961), 577--591.

\bibitem{Oss}
\textsc{R. Osserman}. 
\emph{On the inequality $\Delta u \geq f(u)$}. Pacific J. Math.  \textbf{7}   (1957),  1641--1647. 


\bibitem{PS1999}
\textsc{P. Pucci, J. Serrin \& H. Zou}.
\emph{A strong maximum principle and a compact support principle for singular elliptic inequalities.}
J. Math. Pures Appl. (9)  \textbf{78}  (1999),  769--789. 


\bibitem{PS2000}
\textsc{P. Pucci \& J. Serrin}. 
\emph{A note on the strong maximum principle for elliptic differential inequalities.}  J. Math. Pures Appl. (9)  \textbf{79}  (2000), 57--71. 

\bibitem{PS2001}
\textsc{P. Pucci \& J. Serrin}. 
\emph{The Harnack inequality in $\R^2$ for quasilinear elliptic equations.}
J. Anal. Math.  \textbf{85}  (2001), 307--321.

\bibitem{PS2004}
\textsc{P. Pucci \& J. Serrin}.  
\emph{The strong maximum principle revisited.} J. Differential Equations  \textbf{196}  (2004),   1--66.


\bibitem{PS2007}
\textsc{P. Pucci, M. Rigoli \& J. Serrin}. 
\emph{Qualitative properties for solutions of singular elliptic inequalities on complete manifolds.}
J. Differential Equations  \textbf{234}  (2007),   507--543. 


\bibitem{serrin}
\textsc{J. Serrin}.
\emph{Local behavior of solutions of quasi-linear equations.} 
Acta Math.  \textbf{111} (1964), 247--302.

\bibitem{Talenti}
\textsc{G. Talenti}.
\emph{Elliptic equations and rearrangements.} Ann. Scuola Norm. Sup. Pisa Cl. Sci. \textbf{3}  (1976),  697--718.


\bibitem{Vaz}
\textsc{J.-L. Vazquez}. 
\emph{A strong maximum principle for some quasilinear elliptic equations.} Appl. Math.
Optim. \textbf{12} (1984), 191--202.

\end{thebibliography}
\end{document}